\documentclass[12pt]{amsart}
\usepackage{mathrsfs}
\usepackage{latexsym,amssymb,amsmath,amsfonts,graphicx,epstopdf,subfigure,float}
\usepackage[colorlinks=true,linkcolor=blue,citecolor=red,urlcolor=blue,]{hyperref}
\usepackage{pifont}
\usepackage{cite}
\usepackage{multirow}
\usepackage{appendix}

\def\aff{\mathop{\rm aff}\nolimits}

\newtheorem{thm}{Theorem}[section]
\newtheorem{lem}{Lemma}[section]

\newtheorem{coro}{Corollary}[section]

\newtheorem{defn}{Definition}[section]
\newtheorem{exam}{Example}[section]

\numberwithin{equation}{section}
\newtheorem{remark}{Remark}[section]

\newcommand{\bD}{{\mathcal{D}}}

 \usepackage{color}

\begin{document}

\title[A dimension drop phenomenon of fractal cubes]{A dimension drop phenomenon of fractal cubes}

\author{Liang-yi Huang} \address{College of Computer, Beijing Institute of Technology, Beijing, 100080, China
}
\email{liangyihuang@bit.edu.cn}

  \author{Hui Rao$^*$} \address{Department of Mathematics and Statistics, Central China Normal University, Wuhan, 430079, China
} \email{hrao@mail.ccnu.edu.cn
 }

\date{October 12, 2020}
\thanks {The work is supported by NSFS Nos. 11971195 and 11601172.}


\thanks{* The correspondence author.}

\maketitle

\begin{abstract}
Let $E$ be a metric space. We introduce a notion of connectedness index of $E$, which is the Hausdorff dimension of the union of non-trivial connected components of $E$. We show that the connectedness index of a fractal cube $E$ is strictly less than the Hausdorff dimension of $E$ provided that $E$ possesses a trivial connected component. Hence the connectedness index is a new Lipschitz invariant. Moreover, we investigate the relation between the connectedness index and topological Hausdorff dimension.
\end{abstract}

\section{\textbf{Introduction}}

An \emph{iterated function system} (IFS) is a family of contractions $\{\varphi_j\}_{j=1}^N$  on $\mathbb{R}^{d}$, and the \emph{attractor} of  the IFS is the unique nonempty compact set $K$ satisfying $K=\bigcup_{j=1}^N\varphi_j(K)$, and it is called a \emph{self-similar set} \cite{JEH81}. Let $n\ge 2$ and let $\bD=\{d_1,\cdots,d_N\}\subset\{0,1,\dots,n-1\}^d$, which we call a digit set. Denote by $\#\bD:=N$ the \emph{cardinality} of $\bD$. Then $n$ and $\bD$ determine an IFS $\{\varphi_j(z)=\frac{1}{n}(z+d_j)\}_{j=1}^N$, whose attractor $E=E(n,\bD)$ satisfies the set equation
\begin{equation}\label{fractalcube}
E=\frac{1}{n}(E+\bD).
\end{equation}
We call $E$ a \emph{fractal cube}\cite{XX10}, especially, when $d=2$, we call $E$ a \emph{fractal square}\cite{LLR13}.

There are some works on topological and metric properties of fractal cubes. Whyburn \cite{Why58} studied the homeomorphism classification, Bonk and Merenkov \cite{Bonk13} studied the quasi-symmetric classification. Lau, Luo and Rao \cite{LLR13} studied when a fractal square is totally disconnected. Xi and Xiong \cite{XX10} gave a complete classification of Lipschitz equivalence of fractal cubes which are totally disconnected. Recently, the studies of \cite{YZ18,RuanWang17} focus on the the Lipschitz equivalence of fractal squares which are not totally disconnected.

Topological Hausdorff dimension is a new fractal dimension introduced by Buczolich and Elekes \cite{BBE15}. It is shown in \cite{BBE15} that for any set $K$ we always have $\dim_{tH}K\le\dim_HK$, where $\dim_{tH}$ and $\dim_H$ denote the topological Hausdorff dimension and Hausdorff dimension respectively. Ma and Zhang \cite{MZ20} calculated topological Hausdorff dimensions of a class of fractal squares.

Let $K$ be a metric space. A point $x\in K$ is called \emph{a trivial point} of $K$ if $\{x\}$ is a connected component of $K$. Let $\Lambda(K)$ be the collection of trivial points in $K$. Denote
\begin{equation}\label{cH}
\mathcal{I}_c(K):=\dim_HK\setminus\Lambda(K),
\end{equation}
and we call it the \emph{connectedness index} of $K$. It is obvious that $\mathcal{I}_c(K)\le\dim_HK$. Clearly, the connectedness index is a Lipschitz invariant.
The main results of the present paper are as follows.

\begin{thm}\label{dimdecrease}
Let $E=E(n,\bD)$ be a $d$-dimensional fractal cube. If $E$ has a trivial point, then
$\mathcal{I}_c(E)<\dim_HE$.
\end{thm}

However, Theorem \ref{dimdecrease} is not valid for general self-similar sets, even if the self-similar sets satisfy the open set condition.

\begin{exam}
\emph{Let $Q=\{0\}\cup\left(\bigcup\limits_{k=0}^\infty[\frac{1}{2^{2k+1}},
\frac{1}{2^{2k}}]\right)$. Observe that
$Q=\frac{Q}{4}\cup[\frac{1}{2},1]$ and $\frac{Q}{2}\cup Q=[0,1]$.
Then $Q$ is a self-similar set satisfying the equation
$$
Q=\frac{Q}{4}\cup\left(\frac{Q}{4}+\frac{1}{2}\right)\cup
\left(\frac{Q}{2}+\frac{1}{2}\right).
$$
The set $Q$ has only one trivial point, that is $0$. Therefore, $\mathcal{I}_c(Q)=\dim_HQ=1$. Figure \ref{Q'} illustrates $Q'$, a two dimensional generalization of $Q$. Similarly, $Q'$ is a self-similar set, and the unique trivial point of $Q'$ is $\mathbf{0}$.}

\begin{figure}[H]
  \centering
  \includegraphics[width=5 cm]{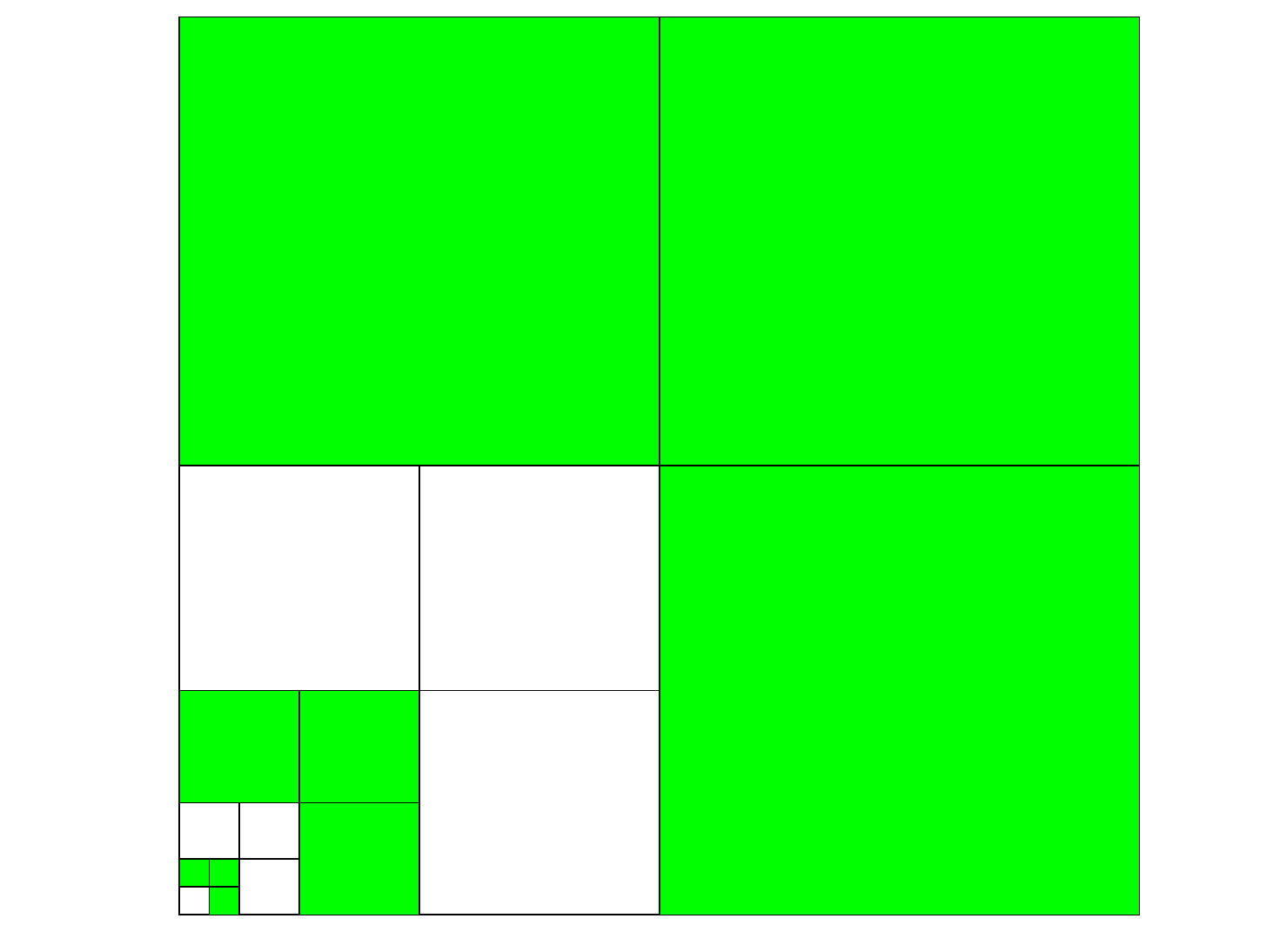}
\caption{The self-similar set $Q'$.}
\label{Q'}
\end{figure}
\end{exam}

Using Theorem 3.7 of \cite{BBE15} we show the following.

\begin{thm}\label{tHlesscH}
For a non-empty $\sigma$-compact metric space $K$, we have $\dim_{tH}K\le\mathcal{I}_c(K)$.
\end{thm}

Zhang \cite{YZ20} asked when $\dim_{tH}E=\dim_HE$, where $E$ is a fractal square. According to \cite{YZ20}, a digit set $\bD$ is called a \emph{Latin digit set}, if every row and every column has the same number of elements (see Figure \ref{fig:latin1}). For a fractal square $E=E(n,{\mathcal D})$, Zhang showed that if $\dim_{tH}E=\dim_HE$, then either $E=[0,1]\times C$, or $E=C\times [0,1]$ for some $C\subset [0,1]$, or ${\mathcal D}$ is a Latin digit set.

\begin{figure}
\subfigure[The digit set of $L$.]{\includegraphics[width=5 cm]{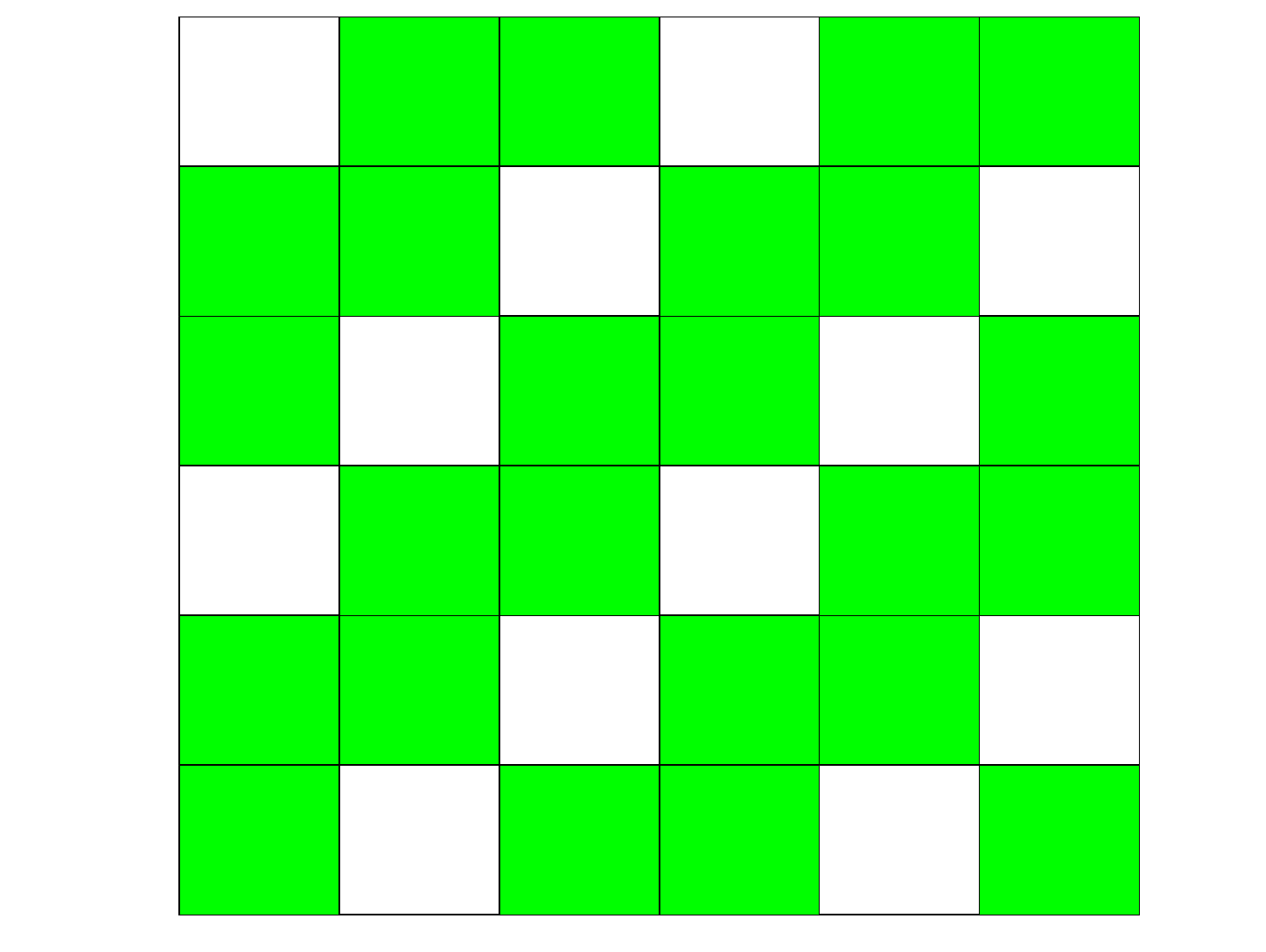} }
\subfigure[The Latin fractal square $L$.]{\includegraphics[width=5 cm]{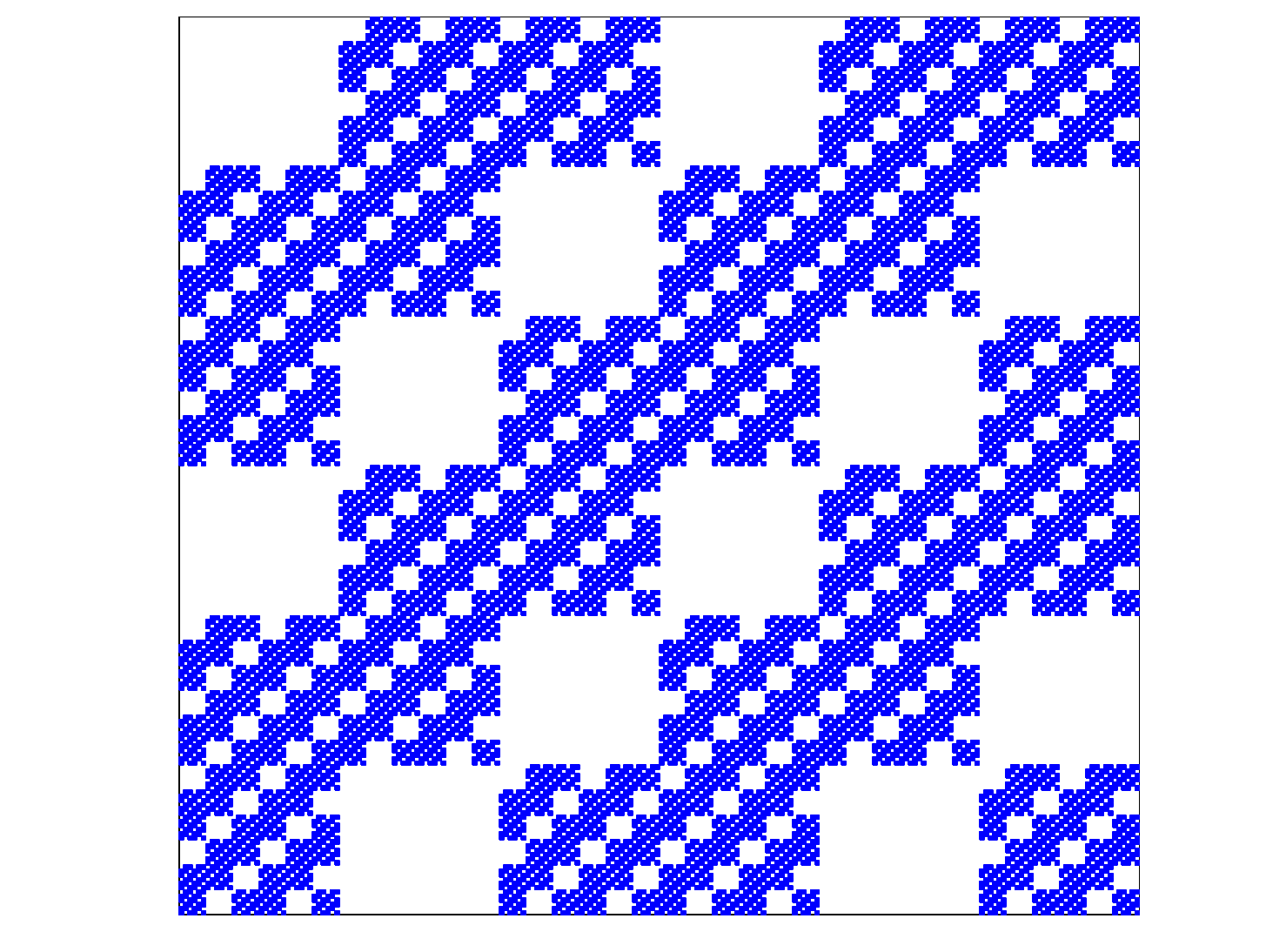}}
\caption{It is shown in \cite{YZ20} that $\log 12/\log 6=\dim_{tH} L<\dim_H L=\log 24/\log 6$. While by Theorem \ref{dimdecrease} and Theorem \ref{tHlesscH}, we directly have $\dim_{tH} L<\dim_H L$.}
\label{fig:latin1}
\end{figure}

As a corollary of Theorem \ref{dimdecrease} and Theorem \ref{tHlesscH}, we obtain a new necessary condition for $\dim_H E=\dim_{tH}E$.

\begin{coro}\label{necessary}
Let $E$ be a $d$-dimensional fractal cube. If $\dim_H E=\dim_{tH}E$, then $E$ has no trivial point.
\end{coro}

\begin{remark}
\emph{Another application of Theorem \ref{dimdecrease} is on the gap sequences of fractal cubes, a Lipschitz equivalent invariant introduced by Rao, Ruan and Yang \cite{RRY08}. For a fractal cube $K$, let $\{g_m(K)\}_{m\ge 1}$ be the gap sequence.
Using Theorem \ref{hasisland} of the present paper, it is proved in \cite{HZR20} that if $K$ has trivial point, then $\{g_m(K)\}_{m\ge 1}$ is equivalent to $\{m^{-1/\gamma}\}_{m\ge 1}$, where $\gamma=\dim_HK$.}
\end{remark}

Finally, we calculate the connectedness indexes of two fractal squares in Figure \ref{kuang}, and illustrate the application to Lipschitz classification.

\begin{exam}\label{examK}
\emph{Let $K$ and $K'$ be two fractal squares indicated by Figure \ref{kuang}. It is seen that $\dim_HK=\dim_HK'=\frac{\log 14}{\log 5}$. By Theorem 1.3 of \cite{MZ20}, one can obtain that $\dim_{tH}K=\dim_{tH}K'=1+\frac{\log 2}{\log 5}$. We will show in section 5 that
$$
\mathcal{I}_c(K)=\frac{\log(8+\sqrt{132}/2)}{\log 5}\quad\text{and}\quad
\mathcal{I}_c(K')=\frac{\log 13}{\log 5}.
$$
So $K$ and $K'$ are not Lipschitz equivalent.}

\begin{figure}[H]
\subfigure[The digit set of $K$.]{\includegraphics[width=4.5 cm]{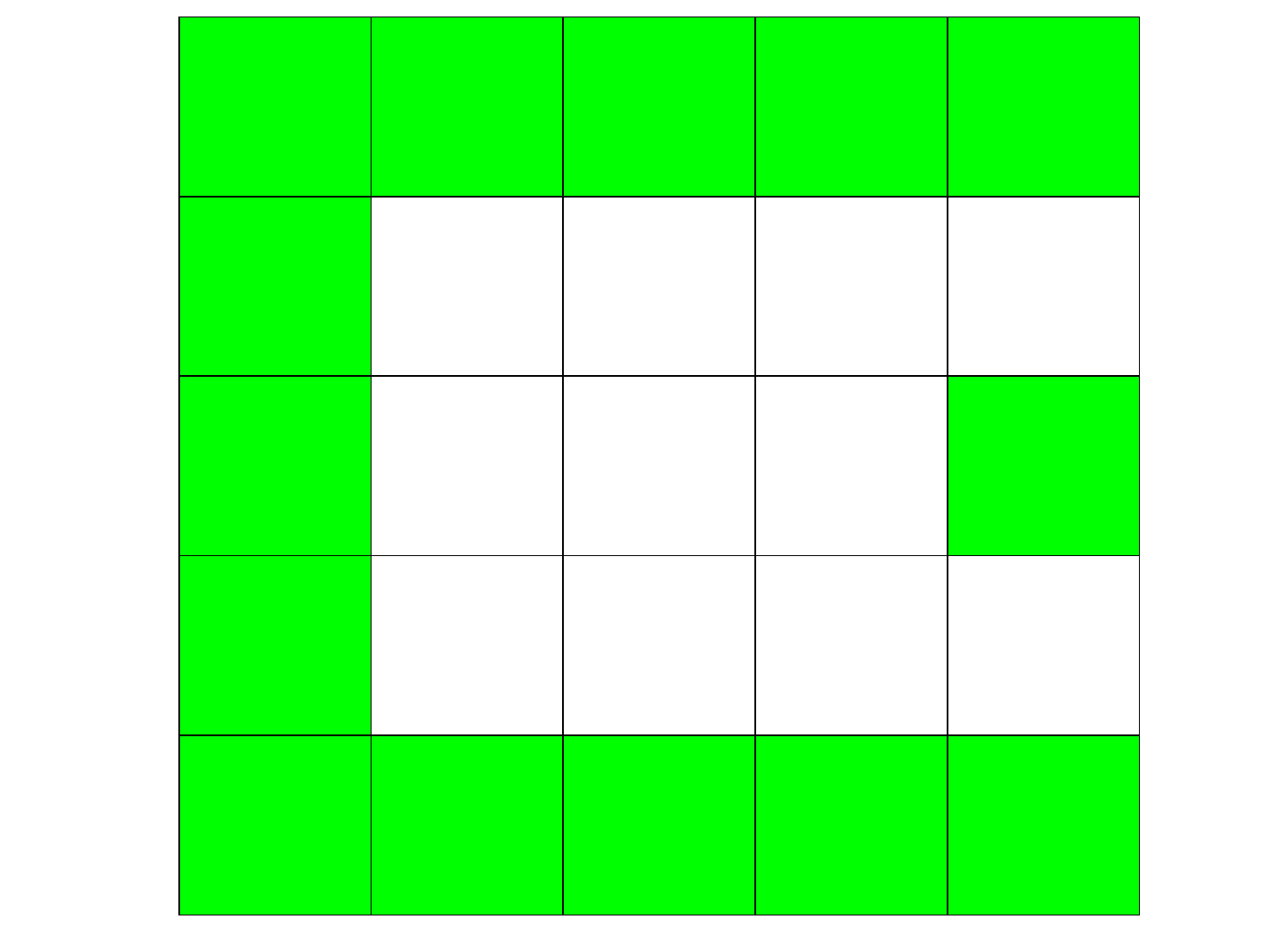} }\quad\quad
\subfigure[The digit set of $K'$.]{\includegraphics[width=4.5 cm]{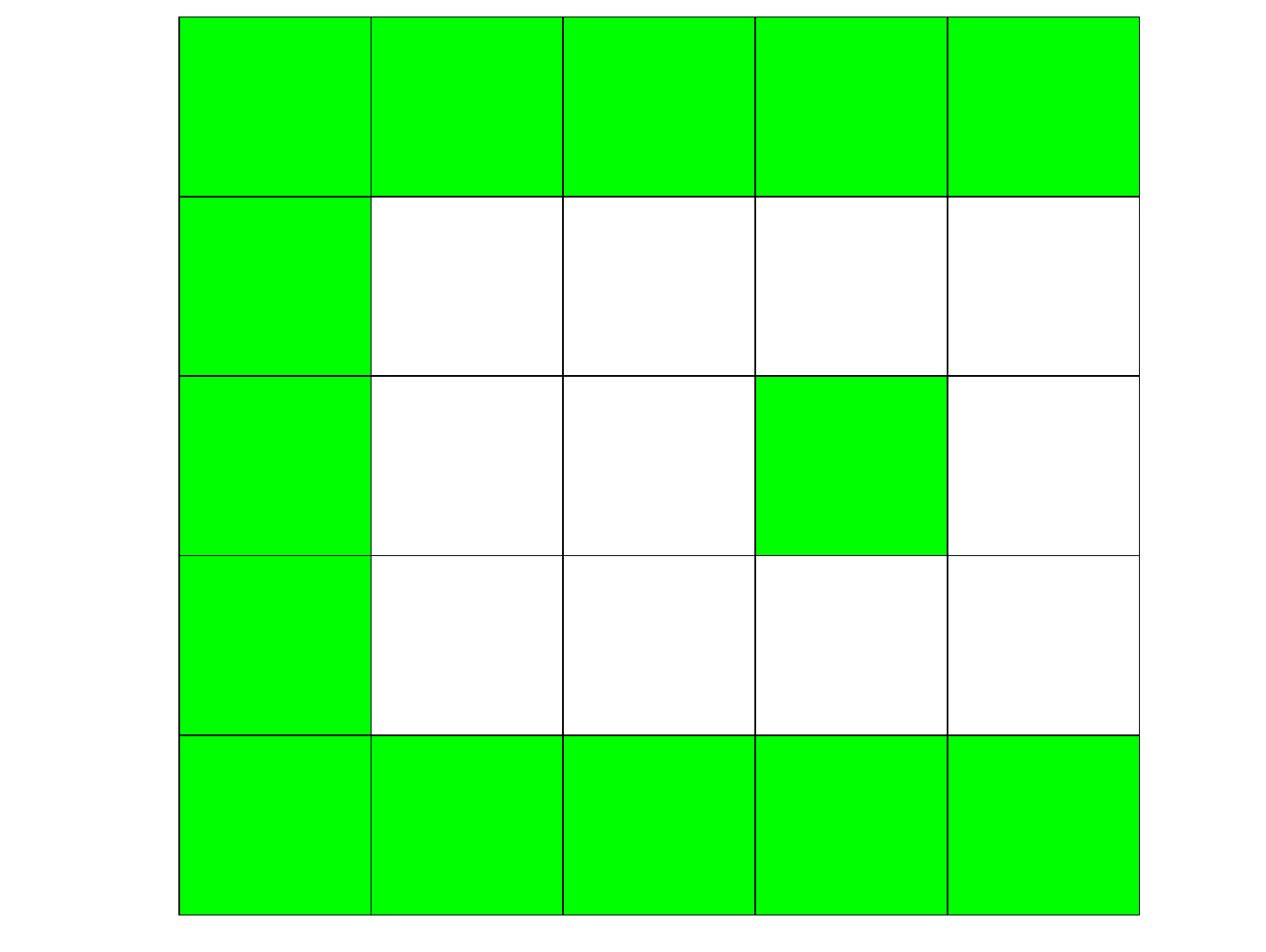}}
\caption{}
\label{kuang}
\end{figure}
\end{exam}

This article is organized as follows. In section 2, we recall some basic facts of $r$-face of the polytope $[0,1]^d$. In section 3, we prove Theorem \ref{dimdecrease}. In section 4, we prove Theorem \ref{tHlesscH}. In section 5, we give the details of Example \ref{examK}.

\section{\textbf{Preliminaries on $r$-faces of $[0,1]^d$}}

We recall some notions about convex polytopes, see \cite{ZM95}. 
Let $C\subset\mathbb{R}^d$ be a convex polytope, let $F$ be a convex subset of $C$. The \emph{affine hull} of $F$, denoted by $\aff(F)$, is the smallest affine subspace containing $F$. We say $F$ is a \emph{face} of $C$, if any closed line segment in $C$ with a relative interior in $F$ has both endpoints in $F$.

The \emph{dimension} of an affine subspace is defined to be the dimension of the corresponding linear vector space. The \emph{dimension} of a face $F$, denoted by $\dim F$, is the dimension of its affine hull. Moreover, $F$ is called an \emph{$r$-face} of $C$, if $F$ is a face of $C$ with dimension $r$. We take it by convention that $C$ is a $d$-face of itself if $\dim C=d$. For $z\in C$, a face $F$ of $C$ is called the \emph{containing face} of $z$ if $z$ is a relative interior point of $F$.

Let $\mathbf{e}_1,\dots,\mathbf{e}_d$ be the canonical basis of $\mathbb{R}^d$. The following facts about the $r$-faces of $[0,1]^d$ are obvious, see Chapter 2 of \cite{ZM95}.

\begin{lem}\label{face}
\indent\emph{(i)} Let $A\cup B=\{1,\dots,d\}$ be a partition with $\#A=r$. Then the set
\begin{equation}\label{Fface}
F=\left\{\sum\limits_{j\in A}c_j\mathbf{e}_{j};\ c_j\in[0,1]\right\}+b
\end{equation}
is an $r$-face of $[0,1]^d$ if and only if $b\in T$, where
\begin{equation}\label{T}
T:=\left\{\sum\limits_{j\in B}\varepsilon_j\mathbf{e}_j;\ \varepsilon_j\in\{0,1\}\right\};
\end{equation}
\indent\emph{(ii)} For any $r$-face $F$ of $[0,1]^d$, there exists a partition $A\cup B=\{1,\dots,d\}$ with $\#A=r$ such that $F$ can be written as \eqref{Fface}.
\end{lem}

We will call $F_0=\{\sum\limits_{j\in A}c_j\mathbf{e}_{j};\ c_j\in[0,1]\}$ a \emph{basic $r$-face related to the partition $A\cup B$}. We give a partition $B=B_0\cup B_1$ according to $b$ by setting
\begin{equation}\label{B0B1}
\begin{array}{ll}
B_0&=\{j\in B;\ \text{the $j$-th coordinate of $b$ is 0}\},\\
B_1&=\{j\in B;\ \text{the $j$-th coordinate of $b$ is 1}\}.
\end{array}
\end{equation}
Let $x=\sum\limits_{j\in A}\alpha_j\mathbf{e}_{j}+\sum\limits_{i\in B}\beta_i\mathbf{e}_i\in[0,1]^d$, we define two projection maps as follows:
\begin{equation}\label{piAB}
\pi_A(x)=\sum\limits_{j\in A}\alpha_j\mathbf{e}_{j},\quad
\pi_B(x)=\sum\limits_{i\in B}\beta_i\mathbf{e}_{i}.
\end{equation}
If $F$ is an $r$-face of $[0,1]^d$, we denote by $\mathring{F}$ the relative interior of $F$.

\begin{lem}[Chapter 2 of \cite{ZM95}]\label{faceprop}
Let $C\subset\mathbb{R}^d$ be a polytope.\\
\indent\emph{(i)} If $G$ and $F$ are faces of $C$ and $F\subset G$, then $F$ is a face of $G$.\\
\indent\emph{(ii)} If $G$ is a face of $C$, then any face of $G$ is also a face of $C$.
\end{lem}

The following lemma will be needed in section 3.

\begin{lem}\label{face-to-face}
Let $F=F_0+b$ be an $r$-face of $[0,1]^d$ given by \eqref{Fface}. Let $u\in\mathbb{Z}^d$. Then $\mathring{F}\cap(u+[0,1]^d)\ne\emptyset$ if and only if $u=b-b'$ for some $b'\in T$, where $T$ is defined in \eqref{T}.
\end{lem}
\begin{proof}
``$\Leftarrow$": Suppose $b'\in T$, then $F-(b-b')=F_0+b'$, and it is an $r$-face of $[0,1]^d$ by Lemma \ref{face} (i). Applying a translation $b-b'$ we see that $F\subset(b-b')+[0,1]^d$, which completes the proof of the sufficiency.

``$\Rightarrow$": Suppose $\mathring{F}\cap(u+[0,1]^d)\ne\emptyset$. Let $z_0$ be a point in the intersection and let $F'=[0,1]^d\cap(u+[0,1]^d)$. Then $F'$ is a face of both $[0,1]^d$ and $u+[0,1]^d$. So we have $F\subset F'$ since $F'$ contains $z_0$, a relative interior point of $F$. Hence $F$ is an $r$-face of $F'$ by Lemma \ref{faceprop} (i). It follows that $F-u$ is an $r$-face of $F'-u$.

Notice that $F'$ is a face of $u+[0,1]^d$, then $F'-u$ is a face of $[0,1]^d$. By Lemma \ref{faceprop} (ii), $F-u=F_0+(b-u)$ is an $r$-face of $[0,1]^d$. By Lemma \ref{face} (i) we have $b-u\in T$.
\end{proof}

\section{\textbf{Trivial points of fractal cubes}}

Let $\Sigma=\{1,2,\dots,N\}$. Denote by $\Sigma^{\infty}$ and $\Sigma^{k}$ the sets of infinite words and words of length $k$ over $\Sigma$ respectively. Let $\Sigma^*=\bigcup_{k\geq0} \Sigma^{k}$ be the set of all finite words. For any $\sigma=\sigma_1\dots\sigma_k\in\Sigma^k$, let $\varphi_{\sigma}=\varphi_{\sigma_1}\circ\dots\circ \varphi_{\sigma_k}$.

In this section, we always assume that $E=E(n,\bD)$ is a $d$-dimensional fractal cube defined in \eqref{fractalcube} with IFS $\{\varphi_j\}_{j\in\Sigma}$. In the following, we always assume that
%

For a point $z\in E$, we say $F$ is the containing face of $z$ means that $F$ is a face of the polytope $[0,1]^d$ and it is the containing face of $z$.

\begin{lem}\label{dim-increase}
Let $z_0\in E$ and $\sigma\in\Sigma^k$ for some $k>0$. Let $F$ be the containing face of $z_0$, let $F'$ be the containing face of $\varphi_{\sigma}(z_0)$. Then either $\varphi_{\sigma}(z_0)\in F$ or $\dim F'\ge\dim F+1$.
\end{lem}
\begin{proof}
Let $A\cup B$ be the partition in Lemma \ref{face} (i) which defines $F$. By the definition of containing face, we have $z_0\in\mathring{F}$. Suppose that $\varphi_{\sigma}(z_0)\notin F$.

Take any point $x\in F\setminus\{z_0\}$ and let $I$ be the closed line segment in $F$ such that $x$ is an endpoint of $I$ and $z_0$ is a relative interior point of $I$.
It is clear that $\varphi_{\sigma}(I)\subset\varphi_{\sigma}([0,1]^d)\subset[0,1]^d$. Since $\varphi_{\sigma}(z_0)\in F'$, we have $\varphi_{\sigma}(I)\subset F'$. By the arbitrary of $x$ we deduce that $\varphi_{\sigma}(F)\subset F'$, hence
\begin{equation}\label{F'dim}
\dim F'\ge\dim \varphi_\sigma(F)=\dim F.
\end{equation}
We claim that $F'$ is not an $r$-face of $[0,1]^d$. This claim together with \eqref{F'dim} imply $\dim F'\ge\dim F+1$.

Suppose on the contrary that $F'$ is an $r$-face of $[0,1]^d$. Then there exists a partition $A'\cup B'=\{1,\dots,d\}$ such that $F'=F'_0+b'$, where $F'_0=\{\sum\limits_{j\in A'}c_j\mathbf{e}_{j};\ c_j\in[0,1]\}$ and $b'\in\{\sum\limits_{j\in B'}\varepsilon_j\mathbf{e}_j;\ \varepsilon_j\in\{0,1\}\}$. Since
$$
\frac{F_0}{n^k}+\frac{b}{n^k}+\varphi_{\sigma}(\mathbf{0})=\varphi_{\sigma}(F)\subset F'=F'_0+b',
$$
we have $F'_0=F_0$. Hence $A'=A$ and $B'=B$. It follows that
\begin{equation}\label{b'}
b'=\pi_B(\varphi_{\sigma}(z_0))=\frac{b}{n^k}+
\pi_B(\varphi_{\sigma}(\mathbf{0}))\in T.
\end{equation}
Notice that
\begin{equation}\label{piB0}
\pi_B(\varphi_\omega(\mathbf{0}))\in\left\{\sum\limits_{j\in B}c_j\mathbf{e}_j;\ c_j\in[0,\frac{n^k-1}{n^k}]\right\}
\end{equation}
for any $\omega\in\Sigma^k$, which together with \eqref{b'} imply that $\pi_B(\varphi_{\sigma}(\mathbf{0}))=\frac{(n^k-1)}{n^k}b$. Hence $b'=b$ and it follows that $\varphi_{\sigma}(z_0)\in F$, a contradiction. The claim is confirmed and the lemma is proven.
\end{proof}

For each $\sigma=\sigma_1\dots\sigma_k\in\Sigma^k$, we call $\varphi_\sigma([0,1]^d)\subset E_k$ a \emph{$k$-th cell} of $E_k$. Denote
\begin{equation}\label{layer}
\Sigma_{\sigma}=\{\omega\in\Sigma^k;\pi_A(\varphi_\omega(\mathbf{0}))
=\pi_A(\varphi_{\sigma}(\mathbf{0}))\}
\end{equation}
and set
\begin{equation}\label{H}
H_{\sigma}=\bigcup_{\omega\in\Sigma_\sigma}\varphi_\omega([0,1]^d).
\end{equation}
Indeed, $H_{\sigma}$ is the union of all $k$-th cells having the same projection with $\varphi_\sigma([0,1]^d)$ under $\pi_A$. From now on, we always assume that
\begin{equation}\label{assu2}
\text{$z_0$ is a trivial point of $E$ and $F$ is the containing face of $z_0$}.
\end{equation}

\begin{lem}\label{another-z1}
Let $k>0$, fix $\sigma\in\Sigma^k$. If $H_\sigma$ is not connected or $H_\sigma\cap F=\emptyset$, then there exists $\omega^*\in\Sigma_\sigma$ such that $\varphi_{\omega^*}(z_0)\notin F$ and it is a trivial point of $E$.
\end{lem}
\begin{proof}
Let $\dim F=r$ and let $A\cup B$ be the partition in Lemma \ref{face} (i) which defines $F$.

We claim that if $H_\sigma\cap F\ne\emptyset$, then there is only one $k$-th cell in $H_\sigma$ which intersects $F$. Actually, since $\varphi_\omega(\mathbf{0})\in[0,\frac{n^k-1}{n^k}]^d$ for any $\omega\in\Sigma^k$, if $\varphi_\omega([0,1]^d)\cap F\ne\emptyset$ for some $\omega\in\Sigma_{\sigma}$, then similar to the proof of Lemma \ref{dim-increase} we must have $\pi_B(\varphi_\omega(\mathbf{0}))=\frac{(n^k-1)}{n^k}b$. On the other hand, $\pi_A(\varphi_\omega(\mathbf{0}))=\pi_A(\varphi_{\sigma}(\mathbf{0}))$, so $\omega$ is unique in $\Sigma_\sigma$. Furthermore, $\varphi_\omega(z_0)\in F$ in this scenario.

By the assumption of the lemma and the claim above, there is a connected component $U$ of
$H_\sigma$ such that $U\cap F=\emptyset$. Let $W=\{\omega\in\Sigma_\sigma;\varphi_\omega([0,1]^d)\subset U\}$. For each $\omega\in W$, write
$$
\pi_B(\varphi_\omega(\mathbf{0}))=\sum\limits_{j\in B_0}\alpha_j(\omega)\mathbf{e}_j+\sum\limits_{j\in B_1}\beta_j(\omega)\mathbf{e}_j
$$
First, we take the subset $W'\subset W$ by
$$
W'=\left\{\omega\in W;\ \sum\limits_{j\in B_0}\alpha_j(\omega)\text{ attains the minimum}\right\}.
$$
Then we take $\omega^*\in W'$ such that
$$
\sum\limits_{j\in B_1}\beta_j(\omega^*)=\max\{\sum\limits_{j\in B_1}\beta_j(\omega);\ \omega\in W'\}.
$$
Since $U\cap F=\emptyset$, we have $\varphi_{\omega^*}(z_0)\notin F$.

Let us check that $\varphi_{\omega^*}(z_0)$ is a trivial point of $E$. To this end, we only need to show that
\begin{equation}\label{star}
\varphi_{\omega^*}(z_0)\notin\varphi_\omega([0,1]^d),
\end{equation}
where $\omega\in\Sigma^k\setminus\{\omega^*\}$. Notice that $\varphi_{\omega^*}(z_0)\in\varphi_{\omega^*}(\mathring{F})$, it is clear that \eqref{star} holds for any $\omega\notin\Sigma_\sigma$. Since $U$ is a connected component of $H_\sigma$, we see that \eqref{star} holds for any $\omega\notin W$.

Now suppose $\varphi_{\omega^*}(z_0)\in\varphi_\omega([0,1]^d)$ for some $\omega\in W$, then $\varphi_{\omega^*}(z_0)\in\varphi_{\omega^*}([0,1]^d)\cap\varphi_\omega([0,1]^d)$. By Lemma \ref{face-to-face} we have
$$
\pi_B(\varphi_\omega(z_0))-\pi_B(\varphi_{\omega^*}(z_0))=
\pi_B(\varphi_\omega(\mathbf{0}))-\pi_B(\varphi_{\omega^*}(\mathbf{0}))
\in\frac{b-b'}{n^k},
$$
where $b'\in T$. By the definition of $B_0$ and $B_1$ in \eqref{B0B1}, we know that the $j$-th coordinate of $b-b'$ is $0$ or $-1$ if $j\in B_0$ and is $0$ or $1$ if $j\in B_1$. According to the choosing process of $\omega^*$, on one hand, we have $\sum\limits_{j\in B_0}(\alpha_j(\omega)-\alpha_j(\omega^*))\ge 0$.
So $\alpha_j(\omega)=\alpha_j(\omega^*)$ for $j\in B_0$, that is to say, $\omega\in W'$. On the other hand, since $\omega\in W'$, we have $\sum\limits_{j\in B_1}(\beta_j(\omega)-\beta_j(\omega^*))\le 0$, which forces that $\beta_j(\omega)=\beta_j(\omega^*)$ for $j\in B_1$. Therefore, $b=b'$ and hence $\omega=\omega^*$. This finishes the proof.
\end{proof}

For $k>0$, denote $\bD_k=\bD+n\bD+\dots+n^{k-1}\bD$. We call $E_k=([0,1]^d+\bD_k)/n^k$ the \emph{$k$-th approximation} of $E$. Clearly, $E_k\subset E_{k-1}$ for all $k\ge 1$ and $E=\bigcap_{k=0}^\infty E_k$. For $\sigma=(\sigma_\ell)_{\ell\ge 1}\in\Sigma^\infty$, we denote $\sigma|_k=\sigma_1\dots\sigma_k$ for $k>0$. We say $\sigma$ is a \emph{coding} of a point $x\in E$ if $\{x\}=\bigcap_{k\ge 1}\varphi_{\sigma_1\dots\sigma_k}(E)$.

\begin{defn}\label{island}
Let $U$ be a connected component of $E_k$, we call $U$ a \emph{$k$-th island} if $U\cap\partial[0,1]^d=\emptyset$.
\end{defn}

\begin{lem}\label{leftarrow}
If $E_k$ contains a $k$-th island for some $k>0$, then $E$ has a trivial point.
\end{lem}
\begin{proof}
Since we can regard $E(n,\bD)$ as $E(n^k,\bD_k)$, without loss of generality, we assume that $E_1$ has an island and denote it by $U$. Write $U=\bigcup_{j\in J}\varphi_j([0,1]^d)$, where $J\subset\Sigma$. We call a letter $j\in J$ a special letter. A sequence $\sigma=(\sigma_i)_{i\ge 1}\in\Sigma^\infty$ is called a special sequence, if special letters occur infinitely many times in $\sigma$.

Let
\begin{equation}\label{tildeF}
P=\{x\in E;~\text{at least one coding of $x$ is a special sequence}\}.
\end{equation}
We claim that every point in $P$ is a trivial point. Let $z\in P$ and let $\sigma=(\sigma_i)_{i\ge 1}$ be a coding of $z$ such that $\sigma$ is a special sequence. Suppose $\sigma_k$ is a special letter, it is easy to see that $z\in\varphi_{\sigma_1\dots\sigma_k}([0,1]^d)\subset\varphi_{\sigma_1\dots\sigma_{k-1}}(U)$ and $\varphi_{\sigma_1\dots\sigma_{k-1}}(U)$ is a connected component of $E_k$ with $diam(\varphi_{\sigma_1\dots\sigma_{k-1}}(U))\le\sqrt{d}/n^{k-2}$. Notice that special letters occur infinitely often in $\sigma$, we conclude that $z$ is a trivial point.
\end{proof}

\begin{thm}\label{hasisland}
Let $E$ be a fractal cube with $\dim\aff(E)=d$. Then $E$ has a trivial point if and only if $E_k$ contains a $k$-th island for some $k\ge 1$.
\end{thm}
\begin{proof}
Let $z_0\in E$ be a trivial point. We claim that there exists another trivial point $z^*\in E\cap(0,1)^d$, that is, the dimension of the containing face of $z^*$ is $d$.

Suppose $F$ is the containing face of $z_0$ with $\dim F=r$, where $0\le r\le d-1$. Let $A\cup B$ be the partition in Lemma \ref{face} (i) which defines $F$. Let $\sigma=(\sigma_\ell)_{\ell\ge 1}\in\Sigma^\infty$ be a coding of $z_0$. Then for each $k>0$, $z_0\in H_{\sigma|_k}\cap F$, where $H_{\sigma|_k}$ is defined in \eqref{H}. We will show by two cases that $E$ contains another trivial point of the form $\varphi_\omega(z_0), \omega\in\Sigma^*$, and it is not in $F$.

\medskip
\textit{Case 1. }$H_{\sigma|_k}$ is not connected for some $k>0$.

By Lemma \ref{another-z1}, there exists $\omega^*\in\Sigma_{\sigma|_k}$ such that $z_1=\varphi_{\omega^*}(z_0)\notin F$ is a trivial point of $E$.

\medskip
\textit{Case 2. }$H_{\sigma|_k}$ is connected for all $k>0$.

Let $p>0$ be an integer such that $C_p$ is the connected component of $E_p$ containing $z_0$ and $diam(C_p)<\frac{1}{3}$. It is clear that $H_{\sigma|_p}\subset C_p$, so we have $diam(H_{\sigma|_p})<\frac{1}{3}$. Since $\dim\aff(E)=d$, there exist $j\in\Sigma$ such that
\begin{equation}\label{eqcapF}
\varphi_j([0,1]^d)\cap F=\emptyset.
\end{equation}
We consider the set $H_{j\sigma_1\dots\sigma_p}$. Let $\Sigma_j=\{i\in\Sigma;\pi_A(\varphi_i(\mathbf{0}))=\pi_A(\varphi_j(\mathbf{0}))\}$. It is easy to see that $H_{j\sigma_1\dots\sigma_p}=\bigcup_{i\in\Sigma_j}\varphi_i(H_{\sigma|_p})$.

If $\#\Sigma_j=1$, then $H_{j\sigma_1\dots\sigma_p}\cap F=\varphi_j(H_{\sigma|_p})\cap F=\emptyset$. If $\#\Sigma_j>1$, we have $\varphi_i(H_{\sigma|_p})\cap\varphi_{i'}(H_{\sigma|_p})=\emptyset$ for any $i,i'\in\Sigma_j$ since $diam(H_{\sigma|_p})<\frac{1}{3}$. Hence $H_{j\sigma_1\dots\sigma_p}$ is not connected. So by Lemma \ref{another-z1}, there exists $\omega^*\in\Sigma_{j\sigma_1\dots\sigma_p}$ such that $\varphi_{\omega^*}(z_0)\notin F$ and it is a trivial point of $E$.

\medskip
Then by Lemma \ref{dim-increase}, the containing face of this trivial point has dimension no less than $r+1$. Inductively, we can finally obtain a trivial point $z^*$ whose containing face is $[0,1]^d$. The claim is proved.

Now suppose on the contrary that $E_k$ contains no $k$-th island for all $k\ge 1$. We will derive a contradiction. Let $z^*\in E\cap(0,1)^d$ be a trivial point. Let $U_k$ be the connected component of $E_k$ containing $z^*$, then we have $U_k\cap\partial[0,1]^d\ne\emptyset$. By the Weiestrass-Balzano property of the Hausdorff metric, there exists a subsequence $k_j$ such that $U_{k_j}$ converge. We denote $U^*$ to be the limit. On one hand, $U^*$ is connected since $U_{k_j}$ is connected for each $k_j$. On the other hand, $z^*\in U^*$ and $U^*\cap\partial[0,1]^d\ne\emptyset$. So $U^*$ is a non-trivial connected component of $E$ containing $z^*$, a contradiction. This together with Lemma \ref{leftarrow} finish the proof of the theorem.
\end{proof}

\begin{proof}[\textbf{Proof of Theorem \ref{dimdecrease}}]
First, let us assume $\dim\aff(E)=d$. Since $E$ contains a trivial point, by Theorem \ref{hasisland}, there exists $k>0$ such that $E_k$ contains a $k$-th island. Without lose of generality, suppose $E_1$ has an island $C$. Write $C=\bigcup_{j\in J}\varphi_j([0,1]^d)$, where $J\subset\Sigma$. Let $P$ be defined as \eqref{tildeF}. It has been proved in Lemma \ref{leftarrow} that every point in $P$ is a trivial point.

We denote $P^c=E\setminus P$. Let $\bD'=\bD\setminus\{d_j;\ j\in J\}$ and let $E'$ be the fractal cube determined by $n$ and $\bD'$. It is easy to see that
$$
P^c=\bigcup_{k=0}^\infty\bigcup_{\sigma_1\dots\sigma_k\in\Sigma^k,\sigma_k\in J}
\varphi_{\sigma_1\dots\sigma_k}(E')\subset\bigcup_{\sigma\in\Sigma^*}\varphi_\sigma(E').
$$
Consequently, $\dim_HP^c\le\dim_HE'=\frac{\log\#\bD'}{\log n}<\dim_HE$. Notice that $E\setminus\Lambda(E)\subset P^c$, we have $\mathcal{I}_c(E)=\dim_HE\setminus\Lambda(E)<\dim_HE$.

Next, assume that $\dim\aff(E)<d$. Then there exist $\alpha=(\alpha_1,\dots,\alpha_d)\in\mathbb{R}^d\setminus\{\mathbf{0}\}$ and $c\in\mathbb{R}$ such that
\begin{equation}\label{xalpha}
\langle x,\alpha\rangle=c,\quad \forall x\in E.
\end{equation}
Without loss of generality, we may assume that $\alpha_1\ne 0$. Since $\frac{x+h}{n}\in E$ for any $x\in E$ and any $h\in\bD$, we deduce that
\begin{equation}\label{halpha}
\langle h,\alpha\rangle=(n-1)c.
\end{equation}
Let $x=(x_1,\dots,x_d)\in\mathbb{R}^d$, we define a map by $\pi(x)=(x_2,\dots,x_d)$.
Denote $\widetilde{\bD}=\{\pi(h); h\in\bD\}$ and let $\widetilde{E}$ be the fractal cube determined by $n$ and $\widetilde{\bD}$. Define $g:\mathbb{R}^{d-1}\rightarrow\mathbb{R}^d$ by
$$
g(x_2,\dots,x_d)=(c-\langle \pi(x),\pi(\alpha)\rangle, \pi(x)).
$$
According to \eqref{xalpha} and \eqref{halpha}, one can show that $E=g(\widetilde{E})$. So we have $\mathcal{I}_c(E)=\mathcal{I}_c(\widetilde{E})$ and $\dim_HE=\dim_H\widetilde{E}$.

Therefore, by the first part of the proof and induction we have $\mathcal{I}_c(E)<\dim_HE$.
This finish the proof.
\end{proof}

\section{\textbf{Application to topological Hausdorff dimension}}

The topological Hausdorff dimension is defined as follows:

\begin{defn}[\cite{BBE15}]\label{def-1}
{\rm Let $X$ be a metric space. The topological Hausdorff dimension of $X$ is defined as
\begin{equation}\label{def-2}
\dim_{tH}X=\inf_{ \text{  ${\mathcal U}$ is a basis of $X$}}
\left ( 1+\sup_{U\in {\mathcal U}} \dim_H \partial U \right ),
\end{equation}
where $\dim_H\partial U$ denotes the Hausdorff dimension of the boundary of $U$ and we adopt the convention that
$dim_{tH}\emptyset=\dim_H \emptyset=-1.$
}
\end{defn}

The following theorem gives an alternative definition of the topological Hausdorff dimension.

\begin{thm}[Theorem 3.7 of \cite{BBE15}]\label{thm:origin}
For a non-empty $\sigma$-compact metric space $X$, it holds that
$$
\begin{array}{rl}
\dim_{tH} X=\min\{
h;~ & \exists S\subset X \text{ such that } \dim_H S\leq h-1\\
 &\text{ and } X\setminus S \text{ is totally disconnected}\}.
\end{array}
$$
\end{thm}

\begin{proof}[\textbf{Proof of Theorem \ref{tHlesscH}}.]
Let $G=X\setminus\Lambda(X)$. Clearly $X\setminus G=\Lambda(X)$ is totally disconnected.
Let $t=\dim_{tH}G$. By Theorem \ref{thm:origin}, for any $\delta>0$, there exists $S\subset G$ such that $G\setminus S$ is totally disconnected, and
$$
\dim_H S+1<t+\delta.
$$
We can see that $X\setminus S=\Lambda(X)\cup(G\setminus S)$ is also totally disconnected; for otherwise there is a connected component of $E$ connecting a point $x\in\Lambda(X)$ and a point $y\in G\setminus S$. Again by Theorem \ref{thm:origin}, $\dim_{tH} X\leq \dim_H S+1<t+\delta$. Since $\delta$ is arbitrary, we have $\dim_{tH} X\le\dim_{tH} G$. Therefore,
$$
\dim_{tH}X\le\dim_{tH}G\le\dim_HG=\mathcal{I}_c(X).
$$
\end{proof}

\section{\textbf{Calculation of $\mathcal{I}_c(K)$ in Example \ref{examK}}}



We identify $\mathbb{R}^2$ with $\mathbb{C}$. Let $n=5$. Let $\bD=\{d_1,\dots,d_{14}\}$ be the digit set illustrated in Figure \ref{kuang} (a), denote $\Sigma=\{1,\dots,14\}$. Let $K$ be the fractal square determined by $n$ and $\bD$, and let $\{\varphi_j=\frac{z+d_j}{5}\}_{j\in\Sigma}$ be the IFS of $K$. Denote
\begin{equation}\label{JXY}
\begin{array}{lcl}
J_{XX}&=&\{j\in\Sigma;\ d_j\in\bD\setminus\{i,2i,3i\}\};\\
J_{XY}&=&\{j\in\Sigma;\ d_j\in\{i,2i,3i\}\};\\
J_{YX}&=&\{j\in\Sigma;\ d_j\in\bD\setminus\{i,2i,3i,4,4+4i\}\};\\
J_{YY}&=&\{j\in\Sigma;\ d_j\in\{i,2i,3i,4,4+4i\}\},
\end{array}
\end{equation}
see Figure \ref{K}. Let
$$
X=\left(\bigcup\limits_{j\in J_{XX}}\varphi_j(X)\right)\cup\left(\bigcup\limits_{j\in J_{XY}}\varphi_j(Y)\right),\quad
Y=\left(\bigcup\limits_{j\in J_{YX}}\varphi_j(X)\right)\cup\left(\bigcup\limits_{j\in J_{YY}}\varphi_j(Y)\right).
$$
Then $X$ and $Y$ are graph-directed sets (see \cite{MW88}). The directed graph $G$ is given in Figure \ref{G}.

\begin{figure}[H]
\subfigure[The first iteration of $X$.]{\includegraphics[width=4.5 cm]{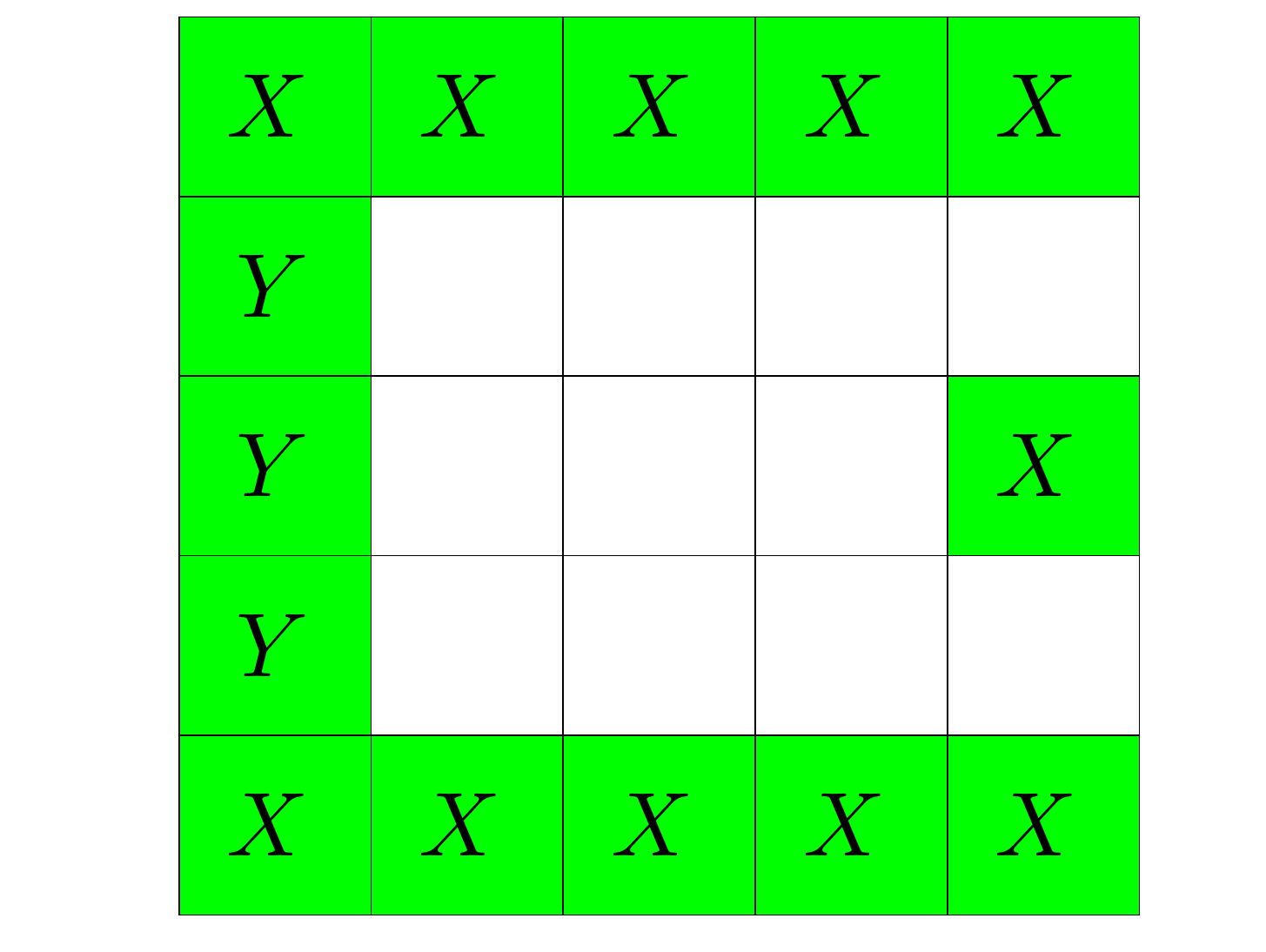}}\quad\quad
\subfigure[The first iteration of $Y$.]{\includegraphics[width=4.5 cm]{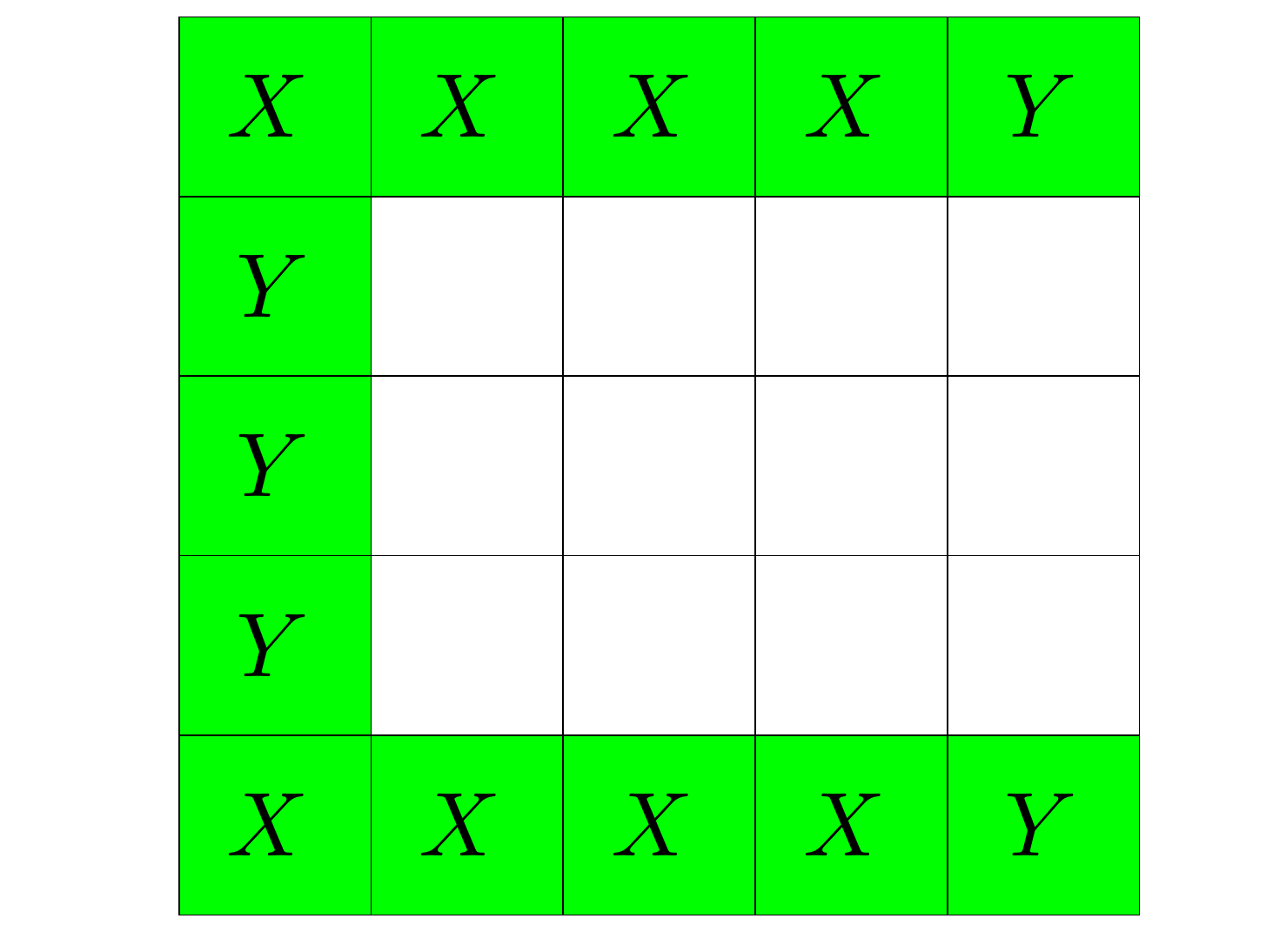}}
\caption{}
\label{K}
\end{figure}

\begin{figure}[H]
  \centering
  \includegraphics[width=6 cm]{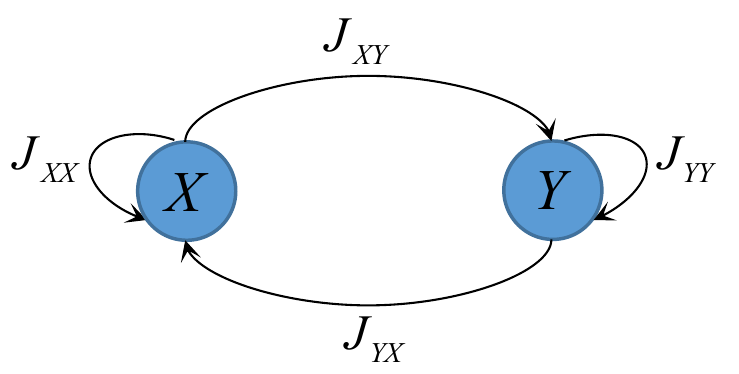}
\caption{The directed graph $G$. Each $d\in J_{XY}$ defined an edge from $X$ to $Y$, and the corresponding map of this edge is $(z+d)/5$. The same hold for $J_{XX}, J_{YX}$ and $J_{YY}$.}
\label{G}
\end{figure}

For each $\ell>0$, let $J_{YX}^{(\ell)}$ be the collection of paths with length $\ell$ which start from $Y$ and end at $X$ in the graph $G$. Similarly, we can define $J_{XX}^{(\ell)},J_{XY}^{(\ell)}$ and $J_{YY}^{(\ell)}$. Let $K_\ell=\bigcup\limits_{\sigma\in\Sigma^\ell}\varphi_\sigma([0,1]^2)$ and $Y_\ell=\bigcup\limits_{\sigma\in J_{YX}^{(\ell)}\cup J_{YY}^{(\ell)}}\varphi_\sigma([0,1]^2)$ be the be the \emph{$\ell$-th approximations} of $K$ and $Y$ respectively. Then $K=\bigcap\limits_{\ell>0}K_\ell$ and $Y=\bigcap\limits_{\ell>0}Y_\ell$.

\begin{lem}\label{lemmaK}
Let $C$ be the connected component of $K$ containing $\mathbf{0}$. Then\\
\indent\emph{(i)} $C=Y$;\\
\indent\emph{(ii)} for any non-trivial connected component $C'\ne C$ of $K$, there exists $\omega\in\Sigma^*$ such that $C'=\varphi_\omega(C)$.
\end{lem}
\begin{proof}
(i) Let $C_\ell$ be the connected component of $K_\ell$ containing $\mathbf{0}$. We only need to show that $C_\ell=Y_\ell$ for all $\ell>0$. Now we define a label map $h$ on the cells in $C_\ell$ as follows. We set $h(\sigma_1\dots\sigma_\ell)=X$ if there exists $\omega_1\dots\omega_\ell\in\Sigma^\ell$ such that
\begin{equation}\label{hX}
\varphi_{\omega_1\dots\omega_\ell}([0,1]^2)=\varphi_{\sigma_1\dots\sigma_\ell}
([0,1]^2)+\frac{1}{n^\ell}\in C_\ell,
\end{equation}
otherwise set $h(\sigma_1\dots\sigma_\ell)=Y$. We will prove by induction that
\begin{equation}\label{YXYY}
\sigma_1\dots\sigma_\ell\in\left\{
\begin{array}{lr}
J_{YX}^{(\ell)}, &\text{if }h(\sigma_1\dots\sigma_\ell)=X,\\
J_{YY}^{(\ell)}, &\text{if }h(\sigma_1\dots\sigma_\ell)=Y.
\end{array}
\right.
\end{equation}

For $\ell=1$, \eqref{YXYY} holds by \eqref{JXY}. Assume that \eqref{YXYY} holds for $\ell$.

\medskip
\textit{Case 1. }$h(\sigma_1\dots\sigma_\ell)=X$.

In this case, \eqref{hX} holds, which means that the right neighbor of $\varphi_{\sigma_1\dots\sigma_\ell}([0,1]^2)$ belongs to $C_\ell$.
If $h(\sigma_1\dots\sigma_\ell\sigma_{\ell+1})=X$, then the right neighbor of $\varphi_{\sigma_1\dots\sigma_\ell\sigma_{\ell+1}}([0,1]^2)$ belongs to $C_{\ell+1}$ and we have $\sigma_{\ell+1}\in J_{XX}$. Hence $\sigma_1\dots\sigma_\ell\sigma_{\ell+1}\in J_{YX}^{(\ell+1)}$. Similarly, if $h(\sigma_1\dots\sigma_\ell\sigma_{\ell+1})=Y$, then $\sigma_{\ell+1}\in J_{XY}$ and $\sigma_1\dots\sigma_\ell\sigma_{\ell+1}\in J_{YY}^{(\ell+1)}$.



\medskip
\textit{Case 2. }$h(\sigma_1\dots\sigma_\ell)=Y$.

In this case, the right neighbor of $\varphi_{\sigma_1\dots\sigma_\ell}([0,1]^2)$ is not contained in $C_{\ell}$. By a similar argument as \emph{Case 1}, we have $\sigma_1\dots\sigma_\ell\sigma_{\ell+1}\in J_{YX}^{(\ell+1)}$ if $h(\sigma_1\dots\sigma_\ell\sigma_{\ell+1})=X$, and $\sigma_1\dots\sigma_\ell\sigma_{\ell+1}\in J_{YY}^{(\ell+1)}$ if $h(\sigma_1\dots\sigma_\ell\sigma_{\ell+1})=Y$.

Therefore, \eqref{YXYY} holds for $\ell+1$. Clearly, \eqref{YXYY} implies that $C_\ell=Y_\ell$. Statement (i) is proved.

(ii) Notice that $\varphi_i(K)\cap\varphi_j(K)\subset C$ for each $i,j\in\Sigma$ with $i\ne j$. Let $C'$ be a non-trivial connected component of $K$. Let $\omega$ be the longest word in $\Sigma^*$ such that $C'\subset\varphi_\omega(K)$. Then $\varphi_\omega^{-1}(C')\subset K$ and there exists $i,j\in\Sigma$ such that $\varphi_\omega^{-1}(C')\cap\varphi_i(K)\cap\varphi_j(K)\ne\emptyset$. It follows that $\varphi_\omega^{-1}(C')\subset C$, hence $C'\subset\varphi_\omega(C)$. Since $C'$ is a connected component, we have $C'=\varphi_\omega(C)$. Statement (ii) is proved.
\end{proof}

By Lemma \ref{lemmaK} we have $\mathcal{I}_c(K)=\dim_HC=\dim_HY=\frac{\log\lambda}{\log 5}$, where $\lambda=\frac{16+\sqrt{132}}{2}$ is the maximal eigenvalue of the matrix
$\begin{bmatrix} 11 & 8 \\ 3 & 5 \end{bmatrix}$.
Let $K'$ be the fractal square in Example \ref{examK}. It is obvious that $\mathcal{I}_c(K')=\frac{\log 13}{\log 5}$.


%
%


\end{document}